\documentclass[11pt]{article}

\usepackage{amsmath,amssymb,amsthm,amstext}

\usepackage{graphicx,epsfig}
\usepackage{fullpage}

\usepackage{color}
\newcommand{\rom}[1]{{\color{blue}[#1]}} 
\renewcommand{\rom}[1]{} 

\newcommand{\ddd}{D}

\renewcommand{\ge}{\geqslant}

\newcommand{\Z}{\mathbb Z}

\newtheorem{theorem}{Theorem}
\newtheorem{corollary}[theorem]{Corollary}
\newtheorem{lemma}{Lemma}
\newtheorem{claim}{Claim}
\theoremstyle{remark}
\newtheorem{definition}{Definition}
\newtheorem{remark}{Remark}

\title{A family of non-periodic tilings of the plane by right golden triangles}
\author{Nikolay Vereshchagin\thanks{The article is supported by Russian Science Foundation (20-11-20203).}\\
HSE University, Russian Federation.}

\begin{document}

\maketitle
\begin{abstract}
We study a family of substitution tilings with similar right triangles of two sizes
which is obtained using the substitution rule introduced
in [Danzer, L. and van Ophuysen, G.
A species of planar triangular tilings with inflation factor $\sqrt{-\tau}$.
Res. Bull. Panjab Univ. Sci. 2000, 50, 1-4, pp. 137--175 (2001)]. In that  paper,
it is proved  this family of tilings can be obtained from a local
rule using decorated tiles. That is, that this family is \emph{sofic}.

In the present paper, we provide an alternative proof of this fact. 
We use more decorated tiles than Danzer and  van Ophuysen (22 in place of 10).
However, our  decoration of supertiles is more intuitive and our local rule is simpler.
  \end{abstract}  

\rom{the abstract is quite long. Maybe it's a matter of taste, but I like abstracts that give in a short paragraph the major results of the paper, thanks to which I can quickly find out where is proven what in my database...}
\section{Introduction}
\rom{My main major comment concerns the structure of the paper. In the current version I
would call it “unconventional” because the paper starts immediately from discussion
of the single family of tilings that is discussed in the paper without introducing a
general context and/or framework. While there is nothing wrong with such approach,
I think that the paper will benefit from a short overview of substitution tilings and
particularly some review of known results or examples of sofic tilings. In the current
version, and only subsection titled “Previous results” discusses these property, but in
my opinion it does not give any context for the current paper. Particularly, I suggest
to refer to the book “Aperiodic Order” by Baake in Grimm for general context and
to the “Tilings Encyclopedia” website for more examples.
Particularly, I suggest to add a new section that introduces general substitution and
sofic tilings and give a short overview of sofic tilings. After that, the current Section 1
may follow with a new title, for example “Tilings with golden right triangles”. That
way the context will be properly introduced and I think that may help the paper
to reach wider audience. Additionally, the new introduction could state the main
results of the paper and refer where the relevant definitions are introduced.}

\paragraph{Tilings and local rules.}
Assume that a finite family of polygons $\{P_1,\dots,P_k\}$,  called \emph{proto-tiles}, is given.
Isometric images of those polygons are called \emph{tiles}.
A \emph{tiling} $T$ is a family of pair-wise non-overlapping tiles, which
means that the interiors of the tiles are disjoint.
\emph{A patch} is a finite tiling.
A patch $P$ is called a  
\emph{fragment} of a tiling $T$, if $P$ is a subset of $T$. 
If the diameter of a patch (the maximal distance between points of its tiles) is at most  
$d$, then we call that patch a \emph{$d$-patch}. 
In a similar way we define {$d$-fragments}.
 
 Local matching rules govern how tiles may be attached to each other in a tiling.
More specifically, a local rule is identified by a positive real $d$ and
by  a set of $d$-patches, whose members are called \emph{illegal patches}.
A tiling \emph{satisfies} the local rule,
if it does not include illegal patches as fragments. 
For instance,  all polygons $P_1,\dots,P_k$ may be unit squares with colored sides, and the local rule may require that 
tiles are attached side-to-side and the colors on the adjacent sides match (the so-called \emph{Wang tiles}). 

\paragraph{Aperiodic tile sets.}
The pair (a set of proto-tiles, a local rule) is called \emph{aperiodic} if 
all tilings of the plane satisfying the local rule are non-periodic and 
such tilings exist.

Aperiodic sets of Wang tiles were used to prove
the undecidability of Berger's \emph{Domino problem}: find an algorithm that given a Wang tile set finds out whether
that set tiles the entire plane~\cite{rob}.

\paragraph{Substitutions.}
The usual scheme to prove non-periodicity is based on the notion of a \emph{substitution}.   
A substitution $\sigma$ is defined by a  similarity ratio $\phi<1$
and a way to cut every polygon $P_i$ 
into a finite number of parts where each part is congruent to some
polygon from the family $\{\phi P_1,\dots,\phi P_k\}$. 
The substitution 
acts on tilings as follows.
Given a tiling, for all $i$,  we cut every tile $P_i$ from the tiling, as  
defined by the substitution. 
We obtain a tiling of the same set by tiles of smaller size.
Then we apply to the resulting tiling some fixed pre-chosen homothety $H$ with the 
coefficient $1/\phi$ to obtain a tiling with initial tiles $P_1,\dots,P_k$. %
That homothety will be called \emph{the reference homothety} in the sequel.
We call the resulting tiling
\emph{the decomposition} of the initial tiling and denote 
it by $\sigma T$.
The inverse operation is called \emph{composition}.
That is, a tiling $T$ is a composition of a tiling $T'$
if $T'=\sigma T$. 

\emph{A supertile} is a tiling, which can be obtained from
an initial tile $P_i$ by applying decomposition several times.
A supertile of the form $\sigma^n P_i$ is called a \emph{supertile of order $n$}.
Thus each supertile of order $n$ consists of several supertiles of order $n-1$.

 Assume that the substitution and the local rule have  the following properties:
 \begin{enumerate}\label{properties}
 \item[P1.] All supertiles satisfy the local rule.
\item[P2.] Every tiling $T$ of the plane satisfying the local rule 
has a unique composition satisfying the local rule (
``the unique composition property'').  This tiling is then denoted by  $\sigma^{-1}T$.
\end{enumerate}
In this case it is not hard to show that 
all tilings of the plane satisfying the local rule are non-periodic and
that such tilings exist.  This can be shown as follows.

\emph{Existence.} 
%
By P1 each supertile satisfies local rule.
Obviously the linear size of a supertile $\sigma^n P_i$
is $(1/\phi)^n$ times larger than that of $P_i$.
Hence there are tiling of arbitrary large parts of
the plane satisfying local rule. By compactness arguments this implies
that there are such tilings of
the entire plane.

\emph{Non-periodicity.}
Assume that a tiling  $T$
satisfying local rule has a non-zero period $\mathbf a$, that is, 
$T+\mathbf a=T$.
Then the vector $\phi\mathbf a$ is the period of $\sigma^{-1}T$. Indeed,
let $H$ denote the reference homotethy. 
The  decomposition of the tiling 
$\sigma^{-1}T +\phi\mathbf a$ is equal to the decomposition
of $\sigma^{-1}T$ shifted by the vector $H \phi\mathbf a=\mathbf a$, that is, to  $T+\mathbf a$.
By our assumption, we have  $T+\mathbf a=T$.
Thus both $\sigma^{-1}T +\phi\mathbf a$ and $\sigma^{-1}T$
are compositions of $T$ and they both satisfy local rule. By P2 
we then have $\sigma^{-1}T +\phi\mathbf a= \sigma^{-1}T$.
Repeating the argument, we can conclude that 
the vector   $\phi^2\mathbf a$ is a period of the tiling
$\sigma^{-2}T $. 
In this way we can construct a tiling
whose period is much smaller than
the linear sizes of tiles,
which is impossible.

This scheme was used to prove aperiodicity of many tile sets. Perhaps,
the most famous example is Penrose--Robinson P4 tilings, 
where the set of proto-tiles consists 
of two isosceles triangles (see \cite{penrose,GS}). 
Other famous examples are Ammann tilings (two L-shaped hexagonal tiles) and 
Ammann--Beenker tilings (a rhombus and a square). 
For the definition of these tilings and for more examples we refer two the textbooks \cite{BG,GS}
and to the Tilings Encyclopedia~\cite{te}.

A similar approach was used to show non-periodicity of the famous Robinson tilings~\cite{rob} with Wang tiles.
Robinson's construction does not fit exactly the described framework, as in that construction supertiles of order $n$ are 
built from 4 supertiles of order $n-1$ and several proto-tiles. However,
for the version of Robinson tilings from the paper~\cite{dls},
the proof of non-periodicity follows exactly the above pattern.
In the tiling of~\cite{dls}, there are $2^{14}$ proto-tiles, which are unit squares, and every tile 
is cut in four smaller squares. 

\paragraph{Substitution tilings.} 
A tiling $T$ is called  a \emph{substitution tiling\footnote{We use here the terminology of~\cite{G}.
Another name for substitution tilings, \emph{self-affine tilings},
was used in~\cite{solomyak}.} associated with substitution $\sigma$}, if 
for each finite subset  $P\subset T$ there is a 
supertile  including $P$. The property P1 implies that 
every substitution tiling satisfies local rule.
In some cases the reverse is also true. We will call this property 
P3:
\begin{itemize}
\item[P3.] Every tiling of the plane  satisfying the local rule is a substitution tiling associated with the substitution $\sigma$.
\end{itemize}
For instance, it happens that the family of Penrose--Robinson tilings coincides with 
the family of substitution tilings associated with the respective substitution.  
The same happens for Ammann A2 tilings, see~\cite{dsv}.

\paragraph{Decorations and sofic tilings.}
Assume now that we are given only a substitution $\sigma$  acting on a set $\{ P_1,\dots, P_k\}$
of polygons and no local rule. Then 
it is natural to ask whether there is a local rule such that the properties P1 and P2 
hold, or a local rule defining the family $L$ of substitution tilings associated with $\sigma$. 
In many cases there is no such local rule. In such cases 
we would like to find a decoration of the  family $\{ P_1,\dots, P_k\}$
and 
a local rule 
for the decorated family of polygons.  
This means the following:
\begin{itemize}
\item each proto-tile is replaced  by a finite number of proto-tiles of the same shape (we think that
they have different colors); 
\item a  local  rule is defined for decorated tiles, let  $\tilde L$ denote the 
family of all tilings of the plane satisfying that local rule; 
\item  the family $L$ coincides with the family of tilings obtained
from tilings  $\tilde T\in\tilde L$ by removing colors. 
\end{itemize}
If such a decoration and local rule exist, then we say that the initial family $L$ is \emph{sofic}.

To prove that the given decoration and local rule 
for decorated tiles satisfy the last item, one usually defines 
a substitution $\tilde \sigma$ for decorated proto-tiles with the following properties:  (1)
each decorated proto-tile is cut exactly in the same way as prescribed by the initial 
substitution $\sigma$
(see an example in Fig.~\ref{shen7}) 
\begin{figure}[t]
\begin{center}
\includegraphics[scale=1]{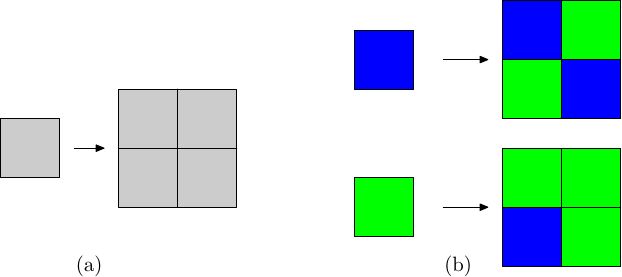}
        \end{center}
        \caption{A substitution (a) and its decoration (b).}\label{shen7}
\end{figure}
and  (2) the properties P1 and P3 hold for $\tilde \sigma$  and the local rule.
To show that it works, assume that there are a decorated substitution $\tilde \sigma$ and a local
rule for decorated tilings with properties P1 and P3. 

In one direction this is trivial: assume that $T$ is obtained from a tiling $\tilde T\in \tilde L$ by removing colors and
assume that $P$ is a finite subset of $T$. Then $P$ is obtained from  a fragment 
$\tilde P\subset \tilde T$ by removing colors.  By P3 the fragment $\tilde P$ occurs in a colored supertile $\tilde\sigma^n\tilde P_i$. Thus
$P$ is included in the supertile  $\sigma^n P_i$.

In the reverse direction: assume that  $T\in L$.
Then every finite $P\subset T$ occurs in a supertile $S$, and by property 
P1 it has a \emph{correct} decoration, which means that  the resulting decorated tiling $\tilde P$ is in $\tilde L$. 
Those decorations for different fragments $P$ may be inconsistent. 
Using compactness arguments, we can show that it is possible to choose consistent such decorations.%
\footnote{Here are more details. Consider the tree whose vertices are correct decorations of fragments 
of the form  $\{F_1,F_2,\dots, F_i\}$. Edges connect a decoration of a fragment 
$\{F_1,F_2,\dots, F_i\}$ to a decoration of the
 fragment  $\{F_1,F_2,\dots, F_i, F_{i+1}\}$ whenever the decorations are consistent. This
 tree has arbitrary long branches. Any vertex of the tree has finitely many neighbors.
 By K\"onig lemma~\cite{konig}, the tree has an infinite branch, which provides 
 a correct decoration of the entire tiling.}

Note that to prove soficness, we do not need property P2 for the decorated family $\tilde L$.
However, if we are not interested in proving that the initial family is sofic
and our goal is just to construct an aperiodic tile set,
we can use P2 (for  $\tilde L$) instead of P3. 

To construct a decoration and a local rule with properties 
P1 and P2, we can try to use a general Goodman-Strauss theorem ~\cite{G} that claims that for every ``good'' 
substitution $\sigma$ 
there is a local rule for decorated tiles with properties P1 and P2.\footnote{Goodman-Strauss formulates property P2, as 
``the local rule enforce the hierarchical structure associated with  $\sigma$'', which means 
that every tiling satisfying the local rule can be uniquely  partitioned into supertiles of order $n$ for each $n$.} 
However, the resulting tile sets are generally gigantic and not really explicit. 
Besides, Goodman-Strauss theorem does not achieve the property  P3.
Both minor points of Goodman-Strauss theorem are inherited by its version by Fernique and Ollinger from~\cite{fo}.
In the case of tilings with Wang tiles, Mozes~\cite{M} proved in a seminal paper that the set of tilings generated by rectangular 
substitutions satisfying a particular property is sofic.

\paragraph{This paper.}
In this paper, we consider tilings with right ``golden'' triangles and the substitution introduced by Danzer and van Ophuysen~\cite{do}. 
Danzer and van Ophuysen showed that 
there is no local rule with properties P1 and P2 for that substitution and 
defined a decoration of the substitution and a local rule for decorated tilings with properties P1, P2 and P3.
However, their decoration of the substitution is  not intuitive
and the local rule is complicated. The local rule uses the notion of a crown.
Let $T$ be a tiling and $A$ a vertex of a tile from $T$. 
\emph{The crown centered at vertex $A$ within $T$} is a fragment of $T$
consisting of all tiles from $T$ that include the point $A$ (not necessarily as a vertex).
The local rule of Danzer and van Ophuysen stipulates that every  crown in the tiling occurs in a supertile.
There are 65 such crowns (up to isometry) and the paper does not even provide their list. 

The goal of the present paper is to provide a more intuitive 
decoration of the substitution and a simpler 
local rule for decorated tilings, also having properties P1, P2 and P3.

\rom{Although there exist some general techniques to show that tiling spaces defined by substitutions are sofic, they are not always easy to implement: on the one hand they are based on some non-trivial assumptions about the substitution itself (very roughly: "the image of a tile is big enough to be able to pass the information without any bottleneck"), on the other hand the resulting sets of tiles are generally gigantic and not really explicit. 

the condition in Goodman-Strauss are more than only side-to-side (complicated notion of "good substitution"). Also, you should cite the other construction of Ollinger "Combinatorial substitutions and sofic tilings" and explain why it does not work (or at last why it does not easily work). I think it is because of the shape of triangular substitution: it is hard to make the required information to reach the corners of the triangle (tiles at corners make a sort of bottleneck).}

\section{Tilings with golden right triangles}

In this paper, we consider a specific substitution and the associated family of   substitution tilings with right ``golden'' triangles introduced by Danzer and van Ophuysen~\cite{do} and later in~\cite{ver}.  

\paragraph{Golden right triangles and tilings.}
The altitude of any right triangle  cuts it into two similar triangles.
Those triangles are denoted by $S,L$ in Fig.~\ref{pic2}(a).
\begin{figure}[ht]
\begin{center}
\includegraphics[scale=1]{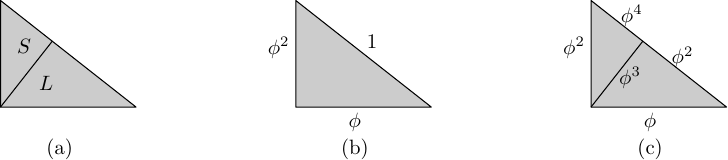}
        \end{center}
        \caption{Golden right triangles.}\label{pic2}
\end{figure}
If the angles of the original right triangle are chosen appropriately,
then the ratio of the size of the initial
triangle to the size of $L$ equals
the ratio of the size of $L$ to the size of $S$.
More specifically, the ratio of the legs
of the initial triangle should be equal to the square root of
the golden ratio $\phi=\sqrt{\frac{\sqrt5-1}2}$.
Such a triangle is shown in Fig.~\ref{pic2}(b). The lengths of the sides
of triangles  $S,L$ are shown in Fig.~\ref{pic2}(c). 
We will call triangles of this shape 
\emph{golden right triangles}.\footnote{The name ``golden'' in a similar context
was used to call isosceles triangles whose all angles are integer multiples of $36^{\circ}$ (Robinson triangles). 
To avoid  confusion, we add the attribute  ``right''.} 

We will use triangles
$L$ and $S$ as proto-tiles. 
More specifically,  isometric images of $L$ are called  \emph{large tiles},
and  isometric images of $S$ are called  \emph{small tiles}.
In Fig.~\ref{pic5}, we can see an example of a tiling. 
\begin{figure}[ht]
\begin{center}
\includegraphics[scale=1]{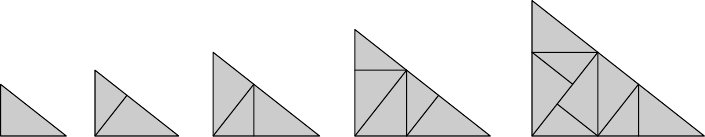}
\end{center}
\caption{A tiling, which is a union of supertiles of orders $0,1,2,3,4$.}\label{pic5}
\end{figure}
We denote by $[T]$ the 
union of all tiles from $T$ and  say that 
$T$ \emph{tiles} $[T]$, or that
$T$ \emph{is a tiling of} $[T]$.

\paragraph{The substitution, decomposition and composition of tilings.}

We consider the following substitution:
\begin{center}
\includegraphics[scale=.7]{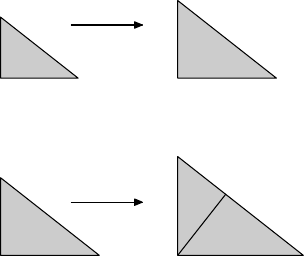}
\end{center}
In the course of decomposition for this substitution, each large tile produces a large and a small tile in
the decomposed tiling and each small tile becomes a large tile.

As usual, we call a tiling $T$ a composition of a tiling $T'$
if $T'$ is the decomposition of $T$. 
There are tilings that have no composition, for instance,
the tiling consisting of a single small tile.
On the other hand, every tiling has at most one  composition (\emph{the unique composition property}). 
As usual, the composition of a tiling $T$ (if exists) is denoted by $\sigma^{-1}T$. 

It may happen that the composition of a tiling again has a composition.
In this case the initial tiling is called \emph{doubly composable}. If
a  tiling can be composed any number of times, we call it \emph{infinitely composable}. 
In terms of \cite G,  infinitely composable tilings are those that have ``hierarchical structure''.
 
\paragraph{Supertiles.}
Recall that \emph{a supertile} is a tiling, which can be obtained from
a small or a large  tile by applying decomposition several times.
Since every large tile is a decomposition of a small tile, 
every supertile can be obtained from a small tile by
applying decomposition some $n$ times.
The number $n-1$ is called then the \emph{order}
of the  supertile. (In particular, the small tile
is a supertile of order $-1$.)
Supertiles of order $i$ are denoted by $S_{i}$. Fig.~\ref{pic5}  shows supertiles of orders $0,1,2,3,4$.

\paragraph{Substitution tilings.}
Recall that a tiling $T$ is called a \emph{substitution tiling}
 if for each finite  $P\subset T$ there is a 
supertile $S$ including $P$.  For instance, all supertiles are substitution  tilings.
There exist substitution tilings of the entire plane. This can be deduced by compactness
arguments from the existence of substitution tilings of arbitrarily large parts of the
plane. However, it is easier to prove this using the following argument.
There are supertiles  of orders 0 and 8, $S_0,S_{8}$, such that
$S_0\subset S_{8}$ and 
$[S_0]$ is included in the interior of $ [S_{8}]$.
Indeed, in Fig.~\ref{fs} we can see a supertile $T$ of order 8.
\begin{figure}[ht]
\begin{center}
\includegraphics[scale=1]{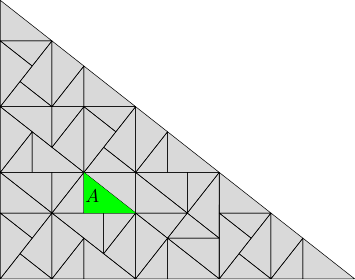}
\end{center}
\caption{The green triangle $A$ is strictly inside a supertile of order 8.}\label{fs}
\end{figure}
The interior of the triangle $[T]$
includes a large tile $A$ (shown in green color).
Applying 8 decompositions to the supertiles
$\{A\}$ and $T$ we get  supertiles $S_{8}=\sigma^8\{A\}$ and $S_{16}=\sigma^8T$,
of orders 8 and 16, respectively. Since $A\in T$, we have 
$S_{8}=\sigma^8\{A\}\subset \sigma^8T=S_{16}$.
In this way we can construct
a tower of supertiles
$$
S_{0}\subset S_8\subset S_{16}\subset S_{24}\subset\dots
$$
where each set $[S_{8n}]$ extends the previous set $[S_{8(n-1)}]$
in all directions.
Therefore the tiling $S_0\cup S_{8}\cup S_{16}\cup \dots$
tiles the entire plane and is a substitution tiling by construction.

It is not hard to see
that every substitution  tiling of the  plane has a composition,
which is again a substitution tiling.
Thus every substitution tiling of the plane  
is infinitely composable. In particular,
every substitution tiling of the plane contains supertiles of all orders.
(To find  a 
supertile of order $n$ in a substitution tiling $T$
of the plane, we can compose it $n$ times and then pick any
large tile in the resulting tiling. The $n$-fold decomposition
of that tile is a supertile of order $n$ and is included in the original
tiling $T$.) As mentioned in the Introduction,
the unique composition property implies
that any infinitely composable (and hence any substitution) tiling
of the plane is non-periodic. 
   
In~\cite{do} and later in~\cite{ver}, it was shown that there is
no local rule such that the family of tilings of the plane
satisfying that local rule coincides with
the family of substitution tilings.
More specifically, it was proved that for any positive
 $d$ there is a periodic 
 (and hence not substitution) tiling  $T_d$ of the  plane, 
 whose all $d$-fragments occur in supertiles. 
For any local rule consisting of $d$-patches,
 either all $d$-patches from $T_d$ are declared legal,
 or  a $d$-patch from $T_d$ is declared illegal.
 In the first case,  the tiling $T_d$ satisfies the local rule.
 In the second case, 
some supertile has an illegal patch, and since all 
 substitution tilings include that supertile, all 
 substitution tilings do not satisfy the local rule.

We will outline a proof of this. We start with the periodic tiling $T$ shown in Fig.~\ref{periodic}.
\begin{figure}[ht]
\begin{center}
\includegraphics[scale=.5]{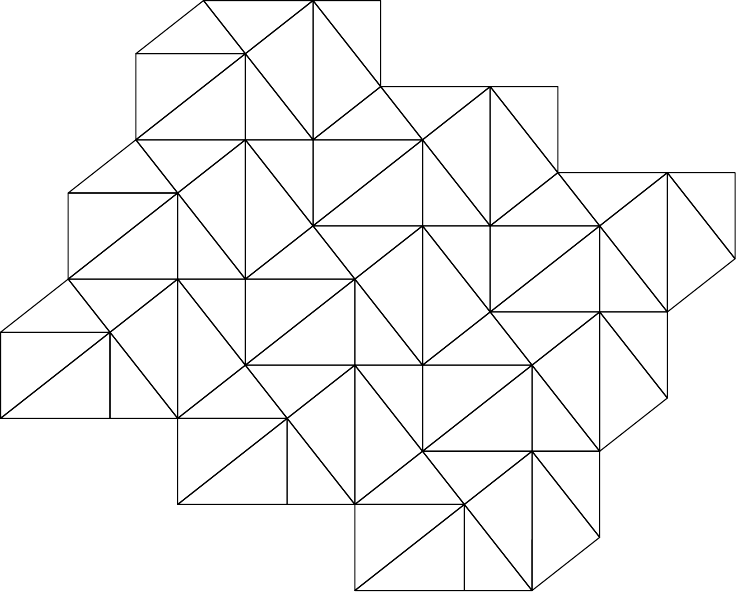}
\end{center}
\caption{Periodic tiling $T$.}\label{periodic}
\end{figure}
It has the following remarkable feature: all its crowns occur in supertiles. In fact, up to an isometry, this tiling
has a single crown,
which occurs in the supertile of order 6 (see Fig.~\ref{fs}).
If $k$ is large enough compared to  $d$, 
then for any $d$-fragment 
 $P$ of the tiling $\sigma^{k}T$ there is a single crown $C$ in $T$ such that $\sigma^k C$ 
includes the entire patch $P$ (this follows from Lemma~\ref{l-bd} below).
As we have seen, the crown $C$ is in the supertile $S_6$ and  hence  $\sigma^k C$ is in $S_{6+k}$.
Thus all $d$-patches in $\sigma^{k}T$ occur in   $S_{6+k}$ and we can let $T_d=\sigma^{k}T$.

\section{Tilings with decorated tiles}

In this section we show that the family of substitution tilings associated with our
substitution is sofic.

\subsection{The local rule of Danzer and van Ophuysen}

We color both tiles in five colors 0, 1, 2, 3, 4.  
The substitution $\tilde\sigma$ acts on decorated tiles as follows:
\begin{center}
\includegraphics[scale=.7]{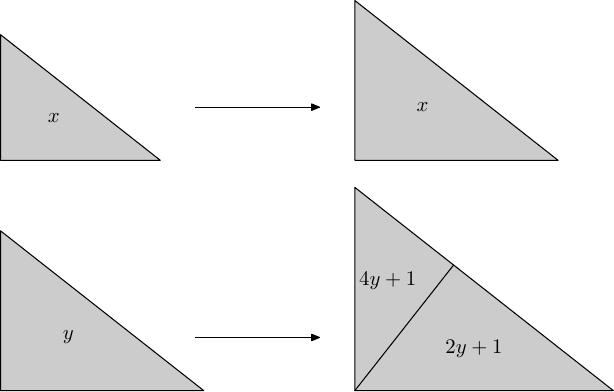}
\end{center}
Addition and multiplication refer to the respective operations modulo 5.

The local rule is based on the notion of a  
\emph{legal crown}. By definition, a crown is legal if it occurs in a supertile.
The local rule stipulates that \emph{all crowns in the tiling must be legal.}
To make this local rule explicit, we need to list all legal crowns.
There are 65 of them (up to an isometry), thus the list is quite long.
However, we can reduce the list using the following observation.

Let us define on tilings the following operation called \emph{shift}.
To shift a tiling by $y=0,1,2,3,4$, we increment the markings of 
all large tiles by $y$ and the markings of 
all small  tiles by $2y$  (modulo 5). It is not hard to see
that the shift of any legal crown is legal. 
Indeed, let $C$ be a crown within a supertile $S_i$ of order $i$, that is obtained from 
a large tile with color $k$. By induction on $i$ it is easy to prove the following:
for each tile $A$ from  $S_i$, its color
is obtained from  $k$ by applying a linear function of the form  
$2^{i+1}x+c_A$ for  small tiles and $2^{i}x+c_A$
for  large tiles. Here $c_A$ denotes a number depending on the location of $A$ within $S_i$.
Thus, if we increase $k$ by $3^{i}y$,
the colors of all large tiles are increased by $y$ and  the colors of all small tiles by $2y$, as $2\cdot 3\equiv1 \pmod 5$.

Let us call two legal crowns \emph{equivalent} if they can be obtained from each other by a shift.
Obviously each equivalence class has 5 legal crowns and hence there are 13 equivalence classes
denoted  $C_1,C_2,C_3,C_4,C_5,C_6,C_7$ and $C'_1,C'_2,C'_3,C'_4,C'_5,C'_6$. 
In Fig.~\ref{dan-stars} we present 
one legal crown from each equivalence class.
\begin{figure}[ht]
\begin{center}
\includegraphics[scale=.6]{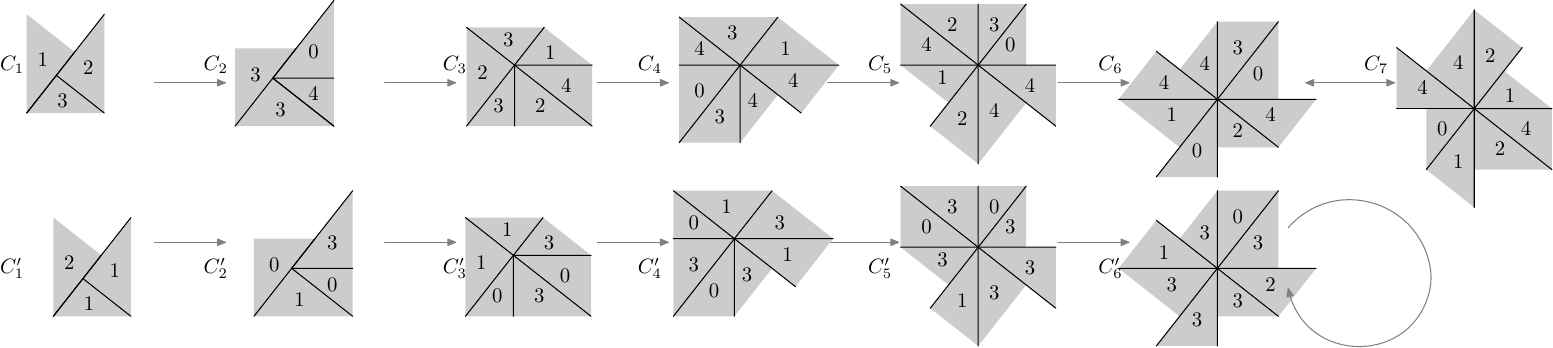}
\end{center}
\caption{Legal crowns for Danzer and van Ophuysen substitution. Every crown
represents 5 crowns obtained by shifts from it. 
Arrows indicate the action of substitution.}\label{dan-stars}
\end{figure}%

This local rule guarantees all the above properties P1, P2 and  P3.

\subsection{Our local rule}
We define first how we color proto-tiles. 
We first choose for every side of the tile its orientation (depicted by an arrow). 
Besides, every side is labeled by an integer number from 
0 to 3. There is the following restriction for those labels:  the 
hypotenuse and the small leg of the large triangle,
and the large leg of the small triangle have even labels, and the remaining sides have odd labels.
Tiles bearing orientation and digital labels on sides are  
called 
\emph{colored tiles}. Each of 2 proto-tiles produces $2^3\cdot 2^3=64$ colored proto-tiles.
Actually, only 22 of these 128 tiles can occur in supertiles and hence we can reduce 
the number of proto-tiles to 22. We will not prove this, since anyway
we obtain more proto-tiles than Danzer and van Ophuysen.

\paragraph{Decomposition and composition of colored tilings.}
The substitution is extended to colored tiles as follows:
\begin{itemize}
\item for small tiles: we increment 
all digital labels by 1 modulo 4 and keep  orientation of all sides
\item for large tiles: we first increment 
all digital labels by 1 modulo 4 keeping  orientation of all sides, and then we 
label the newly appeared altitude by 0 and orient it  
from the foot to the vertex.  The axis of the altitude is divided into two segments, those segments
keep their labels and orientations. 
\end{itemize}
It is not hard to verify that 
the requirement of evenness/oddness of labels is preserved
and thus we obtain again a tiling 
by legally colored tiles.

In Fig.~\ref{pic6} we have shown a large colored tile,  its decomposition, the decomposition of its decomposition
and so on. 
\begin{figure}[t]
\begin{center}
\includegraphics[scale=1]{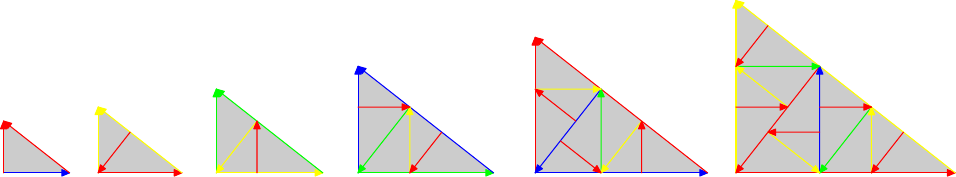}
\end{center}
\caption{Colored  supertiles $S_0,S_1,S_2,S_3,S_4,S_5$.}\label{pic6}
\end{figure}
The digital labels are represented by colors:  red is  0, yellow is  1, green  is 2, blue  is 3,
orientations are shown by arrows.
Long line segments of the same color represent  identically oriented sides with the same 
digital label. This orientation is shown by an arrow at an end of the segment.
The inverse operation is called \emph{the composition of colored tilings}.  

\emph{A colored supertile of order $n$}  is defined as the $n$-fold decomposition of 
a large colored tile.
Colored supertiles of orders 0,1,2,3,4,5
are shown in Fig.~\ref{pic6}. A tiling with decorated tiles is called a \emph{substitution tiling}
if every  its fragment occurs in a (colored) supertile. 

\paragraph{The intuition behind our decoration of supertiles.}
Each supertile $S_i$, $i>0$,  sends a signal whose color equals $(i-1)\bmod 4$ along its altitude from its foot to the top. 
If supertiles $S_{n-1}$ and $S_{n}$ form a supertile $S_{n+1}$,  
then the vertex of the right angle of $[S_{n+1}]$  receives  two signals, $(n-2)\bmod 4$ and $(n-1)\bmod 4$, and in turn sends  the signal  $n\bmod 4$.
The local rule will ensure that all these three signals are ``coherent'', that is, are equal to $(i-2)\bmod 4,(i-1)\bmod 4,i\bmod 4$ for some $i$.
If we were allowed infinitely many colors,
the signal sent by $S_i$ would be just $i-1$. In that case the proof would be much easier.
Each supertile $S_n$ has hierarchical structure, that is, for each $i<n$ it can be partitioned into supertiles $S_i$ and $S_{i-1}$. Hence $S_n$ hosts
many signals. The crucial point is that all those signals are sent along non-overlapping paths. 
We will explain later why the number of colors is 4 (see Remark~\ref{rem4}).

\paragraph{Our local matching rule L.}

To define our local rule, we need a new notion,
similar to that of a crown. We call this  notion  \emph{a star}. 
Let $A$ be a vertex of a tile from a tiling $T$. 
Consider all non-decorated tiles from $T$ that include the point $A$
together with digital marks and orientations
of all the sides that \emph{include the point $A$}.
That information forms \emph{the star within $T$ centered at $A$}.
It is important that we forget orientations and digital labels
of the outer sides of tiles from a star.  
A star may be \emph{incomplete}, which means that no
neighborhood of $A$  is included in the union of tiles from the star. Incomplete stars appear on the borders
of tilings of parts of the plane.
Two examples of stars are shown in Fig.~\ref{pi101}.
\begin{figure}[t]
\begin{center}
\includegraphics[scale=1]{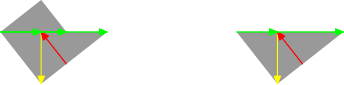}
\end{center}
\caption{A complete star and an incomplete star.  The colors and orientations
of outer sides are not shown, as this information is not included in the star.}\label{pi101}
\end{figure}

\begin{definition}
A complete star is called \emph{legal} if it is one of the stars shown in Fig.~\ref {pic31}. 
The black  line segment on that figure is called
\emph{the axis of a legal star} and may have any orientation and any color. 
All 
digital labels and orientations of all sides lying on the axis must coincide. 
An example of a legal  star is shown in Fig.~\ref{pi101} on the left.
\end{definition}
\begin{figure}[t]
\begin{center}
\includegraphics[scale=1]{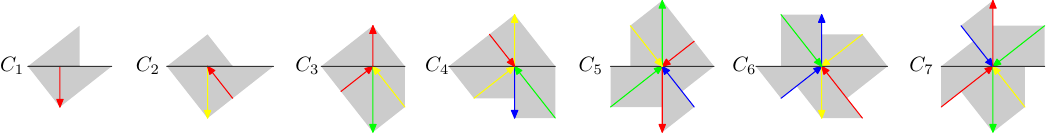}
\end{center}
\caption{The list of legal  stars.
 } \label{pic31}
\end{figure}

 \begin{definition}
A tiling of the  plane  with decorated tiles satisfies our local rule L
if (1) any two sides 
that share a common interval have matching  orientations and digital labels
 and (2) every its star is legal. For a tiling of a part of the  plane  the second item reads: 
every its complete star is legal. 
Tilings  satisfying the local rule L are called \emph{L-tilings}. 
A coloring of tiles in a  tiling with non-decorated  tiles is called \emph{correct}
if the resulting tiling is an  L-tiling.
\end{definition}

\paragraph{Several remarks on legal stars.}
\begin{remark}
We will prove that a star is legal if and only if  it is a complete star within a  supertile.
\end{remark}
\begin{remark}
Observing the tiles in the legal stars we can easily conclude that 
the parity of the digital label of the axis of a legal star must be equal to that of the 
index  of the star. Hence every star from Fig.~\ref {pic31} represents 4 legal stars: there are two ways
to choose  orientation of the axis and two ways to label it.  
Thus there are $7\cdot2\cdot2=28$ different legal stars, up to an isometry.
\end{remark}
\begin{remark}
There are one or two outgoing arrows from the center of any legal star and 
those arrows are orthogonal to the axis of the star. All the remaining arrows are directed towards the center of the star and form with the axis the acute angles $\arcsin\left((\sqrt5-1)/2\right)$ and $\arccos\left((\sqrt5-1)/2\right)$,
called the \emph{smaller} and the \emph{larger} ones, respectively.
Let $n$
denote the index of a legal star. Then the digital labels of the arrows that go into or out of
the center of a legal star are the following. On one side of the star 
the arrow that goes into the center of the star and forms with the axis the smaller acute angle  (if any) is labeled by $n+1$,  the arrow that goes into the center of
the star and forms with the axis the larger acute angle 
  (if any) is labeled by  $n+2$, 
and the outgoing arrow  (if any) is labeled by  $n+3$  (addition modulo 4).
On the other side of the axis the digital labels are   $n-1$,    $n$ and    $n+1$, respectively. 
%
\end{remark}

\section{Results}

The following three theorems claim that our decoration and local rule have the properties P1, P2 and P3. 

\begin{theorem}\label{thm0}
(1) Decomposition of any  L-tiling of the plane is again an L-tiling.
(2) Every supertile is an L-tiling.  (3) Conversely, all legal stars occur in supertiles.
\end{theorem}

 

\begin{corollary}\label{c1}
There exists an L-tiling of the plane.
\end{corollary}

\begin{theorem}\label{l3}
 Any  $L$-tiling of the plane has a composition, which is again 
an L-tiling.
\end{theorem}

\begin{theorem}\label{thm1} 
A tiling of the plane with colored tiles is a 
substitution tiling if  and only if it is an L-tiling.
 \end{theorem} 
 
It follows from  Theorem~\ref{l3}  that all L-tilings of the plane are non-periodic. 
Indeed, they are  infinitely composable, and hence non-periodic, as explained above.
We first prove Theorem~\ref{thm0} and Corollary~\ref{c1}, then we derive Theorem~\ref{thm1} from  Theorem~\ref{l3}
and then we prove the latter.

\begin{proof}[Proof of Theorem~\ref{thm0}]
Let us extend decomposition to stars: to decompose 
a star, we decompose the respective tiling and then delete all the resulting tiles that do not include the center of the star.
It is not hard to verify that the family of legal  stars
is closed under decomposition: see Fig.~\ref{pic3} where
the action of decomposition is shown by grey arrows. 
\begin{figure}[t]
\begin{center}
\includegraphics[scale=1]{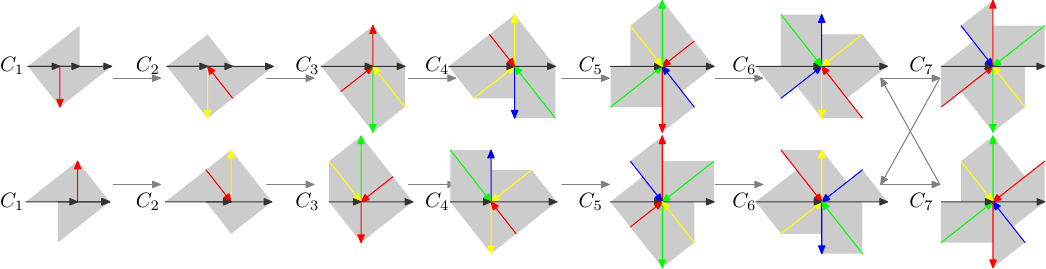}
\end{center}
\caption{The action of decomposition on legal  stars is shown by grey arrows. 
The stars  $C_6,C_7$ in the course of 4 decompositions are mapped to themselves, since the decomposition
works as rotation by the right angle (ignoring colors) on the stars $C_6,C_7$, and 
we increment digital labels modulo 4.
Note that the decomposition maps both stars $C_{5}$ and $C_7$
to $C_6$. This is not surprising, as 
the stars $C_{5}$ and $C_7$ differ only in one tile and
after decomposition  this difference disappears.}\label{pic3}
\end{figure} 

(1) 
Let $T$ be an L-tiling of the plane and $A$ a vertex of a tile from $\sigma T$.
We have to show that the star centered at $A$ in  $\sigma T$ is legal. We consider two cases.

\emph{Case 1:}  $A$ is also a vertex of a tile from $T$. Then the star centered at $A$ in  $T$ is legal,  as $T$ is an L-tiling. 
The star centered at $A$ in  $\sigma T$ is obtained by decomposition from that star and hence is legal as well.
 
\emph{Case 2: } $A$ is not a vertex of a tile from $T$. 
Then $A$ is a foot of the altitude of a large tile $F$ from $T$ and hence lies on the hypotenuse of $F$.
 Let $B,C$ denote endpoints of that  hypotenuse  (see Fig.~\ref{pic54}). Consider the star within $T$ centered at $B$. 
 Observing Fig.~\ref{pic31}, we can see that such situation (the center of the star is an endpoint of the hypotenuse of  
 a large tile and the foot of the altitude of that tile is not a vertex) occurs only in
 stars $C_3$--$C_7$ and in all cases the star includes the triangle $\tilde F$ obtained by the central symmetry 
 \begin{figure}[t]
\begin{center}
\includegraphics[scale=1]{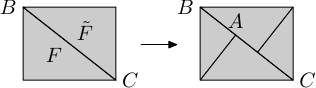}
\end{center}\caption{$A$ is not a vertex of a tile from $T$}\label{pic54}
\end{figure} 
 centered at the middle point of the hypotenuse: $F$ together with $\tilde F$ from a rectangle.  
Decomposition of that rectangle produces two stars $C_1$.

(2) We use induction on the order of the supertile. Base of induction is trivial, since supertiles of order less than 4
have no complete stars. To make the induction step,
we would like to extend item (1) to tilings of parts of the plane. 
However this cannot be done, as there is a tiling of a part of the plane with no complete 
stars such that its decomposition has a complete illegal star (see Fig.~\ref{pic55}).
\begin{figure}[t]
\begin{center}
\includegraphics[scale=1]{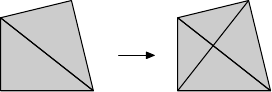}
\end{center}
\caption{An L-tiling of a part of the plane whose decomposition is not an L-tiling}\label{pic55}
\end{figure} 

Why cannot we repeat the above arguments to show that
decomposition of any L-tiling is an L-tiling? 
The only reason why the above arguments fail, is that in Case 2 we 
cannot claim that  the star within $T$ centered at $B$ is legal and hence the tiling $T$ includes 
the large tile $\tilde F$.  
Indeed, it might happen, that although $A$ is an inner 
point of $[T]$, neither  $B$, nor $C$ are  inner points of  $[T]$.
In that case we have no information about the stars  of $T$ centered at $B$ and $C$.

To handle this case, we will show that if $T$ is a supertile, $T=S_i$, then 
\begin{itemize}
\item[(*)]
ignoring orientations and digital marks,
\emph{all} stars in $T=S_i$, including incomplete ones, can be extended to  legal stars by adding some tiles
that do not overlap tiles in $T$.
\end{itemize} 
This statement can be proved by induction on $i$. Base of induction:  the tile $S_0$ has three stars
and they all can be extended to legal stars, provided we ignore digital marks and orientation.\footnote{This ignoring
is important, as there are colored large tiles which possess stars that 
cannot be extended to legal stars.} 
 
 The induction step follows from the fact that the family of tilings $T$
 satisfying (*) is closed under decomposition. Indeed,  assume that a tiling $T$ satisfies (*).
 If $A$ is a vertex in $T$, then the star centered at $A$ in $\sigma T$ can be extended to decomposition of the completion of
the star of $A$ in $T$. Otherwise 
$A$ is a foot of the altitude of a large tile $F$ from $T$ and hence lies on the hypotenuse of $F$. 
Again we 
consider $B,C$, the  endpoints of that  hypotenuse,  see Fig.~\ref{pic54}.
By our assumption, the star centered at $B$ in $T$ can be completed to a legal star (ignoring digital marks and orientation)
by adding tiles that do not overlap $T$.
 As verified above, that star contains the triangle $\tilde F$, as shown in Fig.~\ref{pic54}.
 Hence the star centered at $A$ in $\sigma T$ can be completed to the star $C_1$
 by adding the large tile obtained by decomposition of $\tilde F$ (obviously, that tile does not overlap $\sigma T$).
 
 Now we can handle the hard case. Assume that 
$A$ is a vertex in $S_{i+1}$ but  not a vertex in $S_i$, and $A$ is an inner point of $[S_{i+1}]=[S_{i}]$. 
We have to show that $S_i$ includes the the triangle $\tilde F$ (Fig.~\ref{pic54}).
By (*) the star within $S_i$ centered at $B$ can be extended to a complete legal star $C$ 
by adding some tiles not overlapping $S_i$. As we have seen, $C$ includes the triangle $\tilde F$.
If $\tilde F$ did not belong to $S_i$, then $A$ would lie on the border of $[S_i]$, as $\tilde F$ does not overlap $S_i$.

(3) 
 It suffices to prove the statement for stars $C_1$ only. Indeed, 
for every $i>1$ the star $C_i$ is obtained from $C_1$ by $i-1$ decompositions.
Thus, if $C_1$ occurs in $S_n$, then $C_i$ occurs in  $S_{n+i-1}$.
There are four stars of the type $C_1$: we have two ways to label the axis (yellow or  blue)
and two ways to choose its orientations. The stars $C_1$ with yellow axis
of both orientations appear on the altitude 
of the supertile $S_{10}$ (Fig.~\ref{pic1}).
\begin{figure}[t]
\begin{center}
\includegraphics[scale=.7]{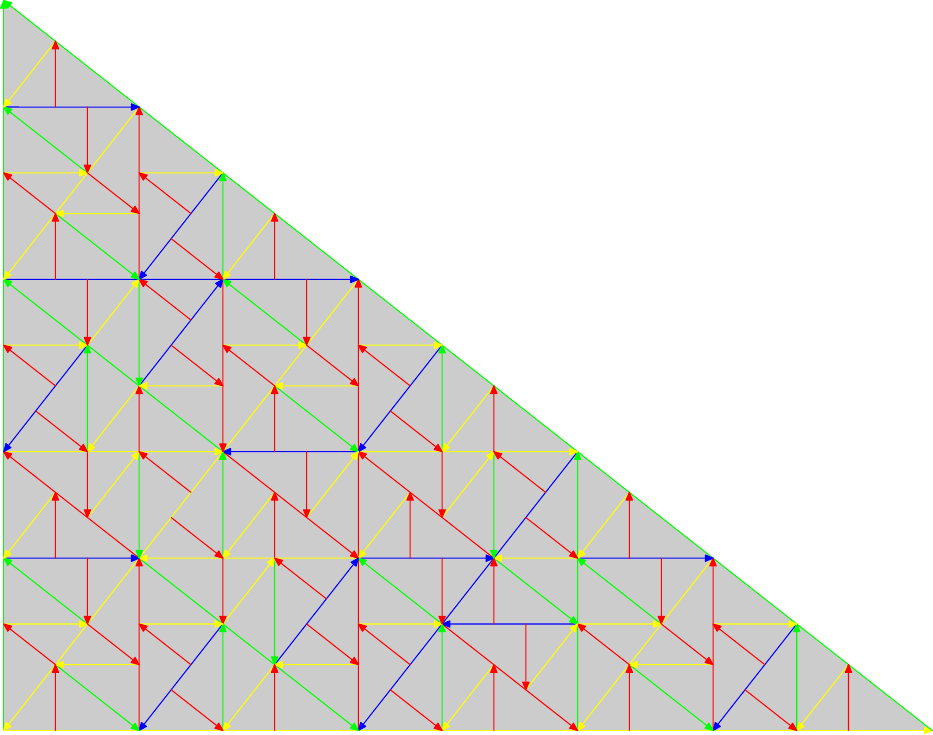}
\end{center}
\caption{A colored supertile $S_{10}$ has all four 
starts of the type $C_1$.}\label{pic1}
\end{figure}
The stars $C_1$ with blue axis appear on the altitude of the supertile $S_{4}$ on 
Fig.~\ref{pic6} (they appear also in Fig.~\ref{pic1}, inside the supertile $S_{10}$).
\end{proof}

\begin{remark}\label{rem4}
We can explain now why we use 4 digital marks. Assume that digital marks belong to  $\Z_k=\{0,1,\dots, k-1\}$,
the substitution $\tilde\sigma$ works as before and the local rule $L_k$ stipulates 
that all stars occur in supertiles. Then we would have lcm$(4,k)$
legal stars of the shape $C_6,C_7$. Indeed, geometrically, substitution acts as $90^\circ$
rotation on stars of this shape (see Fig.~\ref{pic3}). And on digital marks it works as 
adding 1. Thus we have two independent cycles 
of lengths 4 and $k$, whose superposition is a cycle of length     lcm$(4,k)$.

One can show that the choice $k=1,2$ does not work, as for $k=1,2$ there is a periodic tiling 
satisfying the local rule $L_k$.  
The choice $k=3$ might work. However, there are   lcm$(4,3)=12$
legal stars of  the shape $C_6,C_7$ for that $k$, thus the local rule becomes too complicated any way.
So the choice $k=4$ seems to be optimal.
\end{remark}

\begin{proof}[Proof of Corollary~\ref{c1}]
Consider the supertile of order 0 shown in Fig.~\ref{pic6}. 
Decomposing that large tile 8 times,
we obtain a supertile of order 8 which is an L-tiling by Theorem~\ref{thm0} (see Fig.~\ref{pic141}). The green tile $A$ 
\begin{figure}[t]
\begin{center}
\includegraphics[scale=.5]{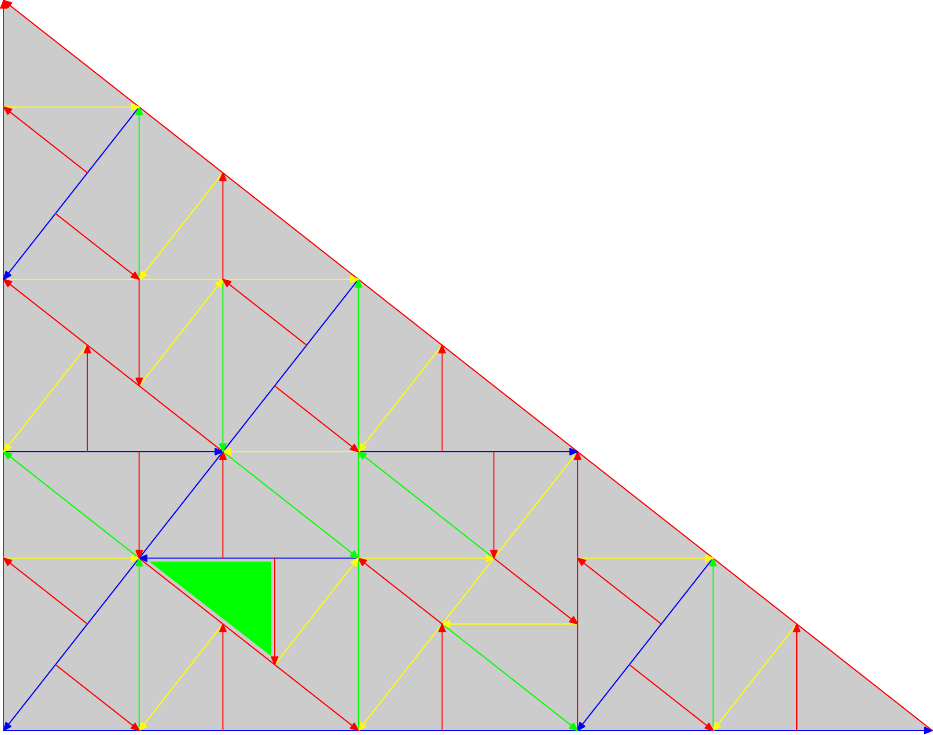}
\end{center}
\caption{A colored supertile $S_{8}$ contains a large tile of the same color as the initial tile $S_{0}$ which it was obtained from.}\label{pic141}
\end{figure}
has the same color as the initial large tile. Thus, if we decompose this supertile $S_8$ eight times, then the green tile
will produce another  supertile $S_8$. In this way we can construct
a tower of supertiles $S_0\subset S_8\subset S_{16}\subset \dots$ whose union is an L-tiling of the entire plane.  
\end{proof}

\begin{remark}
One can show that there are only 7 different small tiles
and 15 different large tiles (up to isometry) that occur in supertiles as inner tiles (a tile is called \emph{an inner 
tile of a supertile}, 
if no its side lies on the border of the supertile). Thus we can reduce the total number of proto-tiles to 22. Indeed, 
to prove  Corollary~\ref{c1}, we do not need remaining tiles.
\end{remark}

\begin{proof}[A derivation of Theorem~\ref{thm1}  from  Theorem~\ref{l3}]
 By definition,
every fragment of every  substitution tiling of the plane occurs  in a colored 
supertile, hence is legal by Theorem~\ref{thm0}(2).  Therefore every  substitution tiling  is an  L-tiling. 
 
To prove the converse, consider any fragment  $P$
of an $L$-tiling $T$ of the  plane. We have to show that $P$ occurs in a (colored) supertile. To this end add in $P$ a finite number of tiles 
from  $T$ so that $P$ becomes an inner part of the resulting fragment $Q$.
By Theorem~\ref{l3} we can compose $T$ any number of times
and the resulting tiling is an  L-tiling.  
Consider the sets of the form $\sigma^k C$ where $C$ is a star within the tiling
$\sigma^{-k}T$. 
As $k$ increases, these sets increase as well. If $k$ is large enough, then the set $Q$ 
is covered by a  single such set, say by $\sigma^k C$,
that is,  $Q\subset \sigma^k C$ 
(Lemma~\ref{l-bd} below). As $\sigma^{-k}T$  
is an L-tiling, all its stars are legal.  In particular, the star $C$
is legal. By Theorem~\ref{thm0}(3), the star $C$ appears in a supertile, say in $S_n$. Therefore
the tiling $\sigma^{k}C$ appears in the supertile $S_{n+k}$. Hence the patch $Q$  
appears in that supertile provided we ignore labels and orientations of its
outer sides. Since no side of the  patch $P$ is an outer side of $Q$,
we are done. It remains to prove Lemma~\ref{l-bd}. 
\end{proof}

\begin{lemma}\label{l-bd} 
If $k$ is large enough compared to  $d$, 
then for any substitution tiling $T$ of the plane 
for any its $d$-fragment 
 $P$ there is a star $C$ in the  tiling $\sigma^{-k}T$ such that $\sigma^k C$ 
includes the entire patch $P$.
\end{lemma}
%
\begin{proof} 
Consider supertiles of the form $\sigma^k \{A\}$ for $A\in\sigma^{-k}T$, call them \emph{$k$-supertiles}.
These supertiles partition $T$ and hence $P$ is covered by a finite 
number of $k$-supertiles. 
More specifically, $P$ is covered by those $k$-supertiles $\sigma^k \{A\}$ which
intersect 
$[P]$. For small $k$, for instance for $k=0$,
the respective tiles $A$ might not belong to a single star within 
the tiling $\sigma^{-k}T$. However, the sizes of $k$-supertiles increase as $k$ increases, 
and for a large enough $k$ 
 the respective tiles $A$ belong to a single star within 
the tiling $\sigma^{-k}T$.
Indeed, cover the set $[P]$ 
by a disc $S$ of radius $d$ (centered at any point from  $[P]$).
It suffices to show that if $k$ is large enough, then there is a star
$C$ in $\sigma^{-k}T$ such that $[\sigma^k C]$ covers disc $S$.
In other words, $[C]$ covers $H^{-k}S$, the inverse image of 
$S$ under the $k$th power of the reference homothety $H$.
The radius of $H^{-k}S$ equals $\phi^kd$,
therefore the claim follows from the following
\begin{quote} 
Geometrical observation:
\emph{
 Let $ \alpha$ denote the minimal angle of the right golden triangle
and $h$ the length of the altitude of the small right golden triangle.
Let $S$ be a disc of diameter $\ddd$.
If $h\ge \ddd/\sin\alpha +\ddd$, then every tiling of the plane has  
a star $C$ such that $S\subset [C]$.}
\end{quote}
\begin{proof}[Proof of the observation]
We have to show that tiles intersecting the disc $S$ belong to a single star.
If there is a single such tile, then this is obvious.
If there are exactly  two such tiles, $A$ and $B$, then they 
must share a part of a side and at least one end of these two sides
belongs to both tiles. Then for the star $C$ centered at that end
we have $[C]\supset A\cup B\supset S$. 
Finally, if there are  three or more such tiles, 
then at least one of those tiles, call it $F$, has common points with $S$ lying on two different
sides of the tile $F$.  Let $E$ denote the common point of those sides 
and let $A,B$ denote the points from $S$ that belong to different sides of $F$.
The angle $\angle AEB$ is one of the angles of the right golden triangle
and the length of $AB$ is at most    $\ddd$. Hence $\ddd\ge |AB|\ge |AE| \sin \alpha$.
Therefore  $|AE|$ is at most $\ddd/\sin\alpha$.
All the points from $S$ are at distance at most $\ddd$ from $A$ and hence at distance at most 
 $\ddd/\sin\alpha+\ddd$ from $E$.  That is, $S$ is covered by 
 the disc with center $E$ and radius $\ddd/\sin\alpha+\ddd$. That disc is covered by the star centered at $E$, 
provided the length of the altitude $h$ of small tiles is at least  its radius $\ddd/\sin\alpha+\ddd$. 
\end{proof}

This observation provides the relation between $k$ and $d$ we need. Assume that $h\ge 2d\phi^k(1/\sin\alpha +1)$. Then 
any $d$-fragment $P$ of the initial substitution tiling $T$ is covered 
by a disc of diameter $2d$ and is included in $\sigma^k C$ for some star $C$ from the tiling $\sigma^{-k}T$.
\end{proof}

\begin{proof}[Proof of Theorem~\ref{l3}]  
  Let $T$ be an  L-tiling of the  plane. We have to show that it has a composition
  and that its composition is again an L-tiling.

\emph{Why $T$ has a composition?} 
Let  $S$ be any small tile from $T$. Consider the star within $T$ centered at the 
vertex of the right angle of $S$. We know that that star is legal.
There are only two stars in the list of legal stars, whose center is a vertex of the right angle of a small tile, the stars $C_{1}$ and $C_{3}$.
\begin{center}
\includegraphics[scale=1]{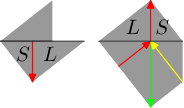}
        \end{center}
In both these stars the small tile
$S$ has the complement labeled $L$ on the picture.
Therefore the      tiling $T$ has a composition which is obtained 
by removing the common legs of all complementary pairs of
tiles $S,L$ and by decrementing all labels by  1 (and then applying $H^{-1}$, where $H$ is the reference homothety).
Hypotenuses of large tiles are built of two legs of different 
tiles from $T$. Those legs have the same labels, since they belong to axes of stars $C_{1}$ and $C_{3}$.
Hence hypotenuses of all large tiles obtain unique digital labels.  
Finally, new labeling of sides 
is 
legal, that is, the hypotenuses and small legs of large tiles and
large  legs of small tiles have even colors and all other sides have odd colors.

\emph{Why the composition of $T$ is an L-tiling?}
We have to show now that the resulting colored tiling is an  $L$-tiling.
Since $T$ satisfies item (1) of the local rule, so does  $\sigma^{-1}T$. Let us verify item (2) of the local rule.

Let $A$ be a vertex of a 
tile from $\sigma^{-1}T$.
We have to prove that the colored star centered at $A$ in  $\sigma^{-1}T$ is legal.
First note that the star centered at the same vertex $A$ in the initial tiling $T$ is different from  
$C_1$, as the centers of stars 
$C_1$ become inner points of sides  in  $\sigma^{-1}T$. Thus that star is one of the stars 
$C_2$--$C_7$. 
We claim that the composition transforms 
these stars by the inverse arrows from Fig.~\ref{pic3}. To prove this, we need the following  
\begin{lemma}\label{l1}
If  $T$ is an  L-tiling of the plane, then every its 
star of type $C_2$--$C_7$, depending on its index,  includes all tiles marked green in Fig.~\ref{star1}
and does not include tiles marked red (the star itself is marked grey).
\end{lemma}
\begin{figure}
\begin{center}
\includegraphics[scale=1]{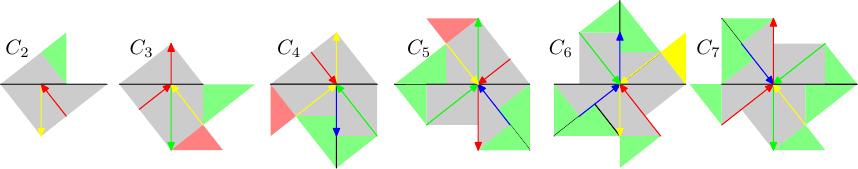}
\end{center}
\caption{Surroundings of legal stars}\label{star1}
\end{figure}
We will first finish the proof of the theorem assuming Lemma~\ref{l1}.
Lemma~\ref{l1} guarantees that via composition tiles from each star  $C_i$ except $C_6$ are transformed 
to tiles from the star $C_{i-1}$. In the course of composition, we decrement the labels and do not change
the orientation of sides. Therefore, digital labels and orientations of all sides 
become, as in the star $C_{i-1}$. Hence for $i>1$, $i\ne 6$, the star 
$C_i$ is transformed to the star $C_{i-1}$. 
For the star $C_6$, one its large tile can be transformed in two ways, depending on 
whether the small yellow  tile is in $T$ or not. If it is, then  tiles from $C_6$ form the star  $C_7$ and otherwise $C_5$.

Hence in the course of composition,  all stars
in the tiling $T$  are transformed by the inverse grey arrows from Fig.~\ref{pic3},
which implies legality of all stars in  $\sigma^{-1}T$.
It remains to prove Lemma~\ref{l1}.
\end{proof}

\section{Proof of Lemma~\ref{l1}}
We first show that in any L-tiling of the plane every star must have some fixed neighborhood,
called \emph{the neighborhood of the star}.  Those neighborhoods include all green small tiles (Fig.~\ref{star1})
except one small tile near $C_7$.
In this analysis, we do not use labels and orientation of sides of tiles.
It is instructive, for reader's convenience, to print out all the legal stars and their neighborhoods (Fig.~\ref{pic32}, \ref{stars-neccessary10} 
and \ref{stars-neccessary1} on pages~\pageref{pic32} and~\pageref{stars-neccessary1})
and then to cut them out of paper. Matching tiles from 
those paper stars with the tiles from the figures below, it is easy to verify all the claims
that  certain stars do not fit in certain places.

\paragraph{The neighborhoods of legal stars.}
The neighborhoods of the stars $C_1,C_{2},C_{3},C_{4},C_{5}$
are shown in Fig.~\ref{stars-neccessary}. They all are centrally symmetric. 
The initial star is colored in grey, the added tiles are colored in light-grey. These neighborhoods are obtained
from each other by decomposition.
One can verify that each of the first five stars indeed must have such neighborhood
as follows.

\emph{The star $C_1$}. Look at the blue vertex inside the grey star $C_1$ 
(Fig.~\ref{stars-neccessary}). That vertex lies on the large  leg of a large tile.
One can easily verify that there is the unique legal star whose center 
lies on the large  leg of a large tile, namely the star 
$C_1$. Hence the star within $T$ centered at the blue vertex is again $C_1$ and we
get the sought neighborhood.
  \begin{figure}[ht]
  \begin{center}
\includegraphics[scale=1]{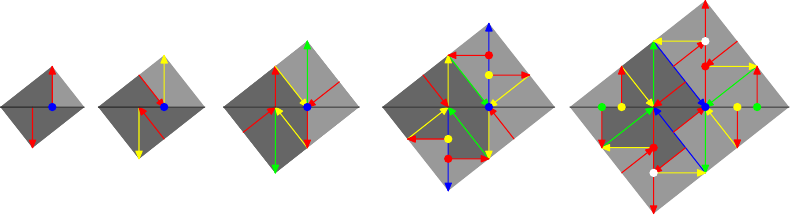}
        \end{center}
        \caption{The neighborhoods of the first five stars.}\label{stars-neccessary}
 \end{figure}     
   
\emph{The star  $C_2$}. 
The argument is similar to the previous one. The star 
$C_2$ has the following feature: it has a vertex  (colored in blue)
that lies on the hypotenuse of its large 
tile. It is easy to verify that there is the unique star
whose center lies on the hypotenuse of its large 
tile, namely the star $C_2$.

\emph{The star $C_3$}. 
The star  $C_3$ is the unique star that has two
small triangles sharing the small leg. Hence the star centered at the  blue vertex is again
$C_3$.

        \emph{The star $C_4$}. 
The star  $C_4$ is the unique star
that has two large triangles sharing the small leg. Hence the star centered at the blue vertex is again
$C_4$. However this star does not complete the neighborhood: 
the stars centered at yellow vertices must be 
$C_1$ and the stars centered at red    vertices again must be $C_1$.

\emph{The star  $C_5$}. 
The star  $C_5$ is the unique star
that has two small  triangles sharing the hypotenuse.
Hence the star centered at the  blue vertex again must be
$C_5$.  The stars centered at yellow and green
vertices must be
$C_1$. Furthermore, the stars centered at red and white vertices must be
$C_2$.

The neighborhoods of the stars
$C_6$, $C_7$ are shown on
Fig.~\ref{stars-neccessary2} (on the right).
 One can verify in the following way that the stars $C_6$, $C_7$ 
 indeed must have such neighborhoods.

\emph{The star $C_6$.} 
The stars centered at yellow and green vertices 
must be  $C_1$ and the stars centered at red and black vertices must
be  $C_2$  (on the left in Fig.~\ref{stars-neccessary2}).
\begin{figure}[ht]
\begin{center}
\includegraphics[scale=1]{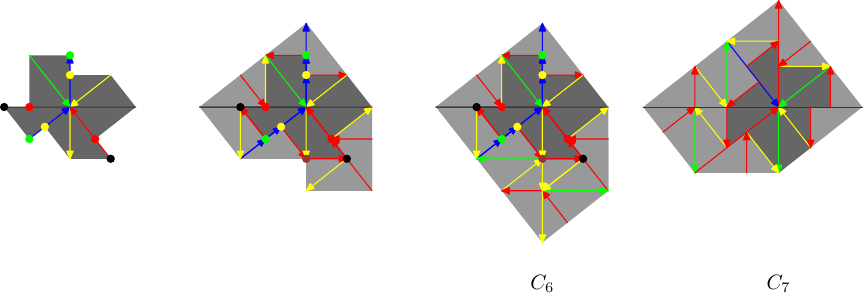}
\end{center}
\caption{The neighborhood of the stars $C_6,C_7$.}\label{stars-neccessary2}
 \end{figure}
Now we see that the star centered at the brown vertex must be  $C_3$,
which is added together with its neighborhood.

\emph{The star  $C_7$}. 
Since the star  $C_7$ can be obtained from $C_6$ by rotation (ignoring labels and orientation),  
the arguments are entirely similar to those for the star $C_6$.


Now we can start the proof of the lemma.
Assume that the star centered at a vertex
$A$ in an L-tilling  $T$ is   $C_i$ where $i>1$.
We have to prove that $T$ includes all tiles marked green on~Fig.~\ref{star11} (= Fig.~\ref{star1})
and does not include tiles marked red (the star itself is marked grey).
\begin{figure}
\begin{center}
\includegraphics[scale=1]{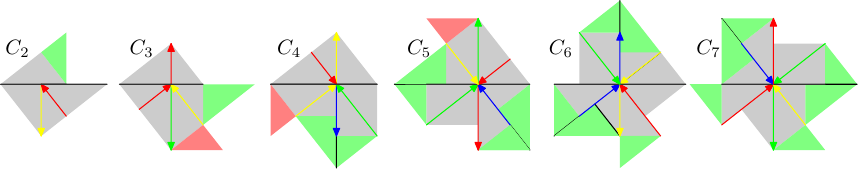}
\end{center}
\caption{Neighborhoods of stars}\label{star11}
\end{figure}
We will treat all $i$'s separately.
We start with simple cases $i=2,5,6$.

\paragraph{The star of
  $A$ in  $T$ is  $C_{2}$ or $C_6$.}   
It is easy to verify that in both cases  the neighborhood of the star
 includes all tiles marked green in Fig.~\ref{star11}, and we are done.

\paragraph{The star of the vertex  $A$ in $T$ is $C_{5}$.}
Fig.~\ref{pi510}(a,b) show the star $C_{5}$
and its neighborhood. We can see that the neighborhood includes all 4 tiles marked green in Fig.~\ref{star11}.
We need to show that the small red triangle is not in $T$. 
For the sake of contradiction assume
that the tiling $T$ includes the patch shown in Fig.~\ref{pi510}(c).
\begin{figure}[ht]
\begin{center}
\includegraphics[scale=1]{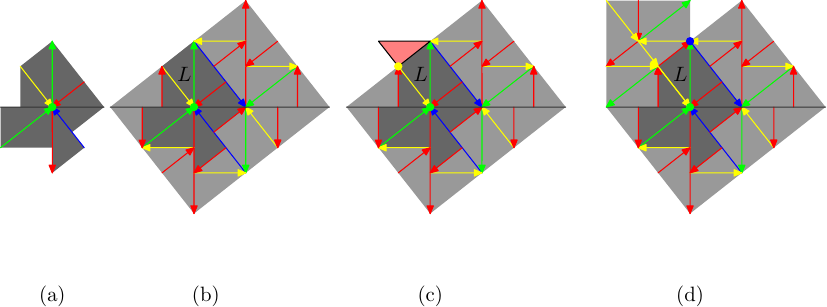}
\end{center}
\caption{Composition of the star $C_5$}\label{pi510}
\end{figure}
We will
say that 
\emph{a star $C_i$ fits for a given patch in a given its vertex}
if one can draw (an isometric copy of) $C_i$  centered at that vertex
so that each of its tiles either does not overlap all the tiles from that tiling, or belongs to that tiling.
It is easy to verify that only the star  $C_3$
fits for this patch in the yellow vertex. Adding to the
patch that star and its neighborhood, we obtain the patch shown in Fig.~\ref{pi510}(d). 
Now, 
we can verify that no legal star
fits for this patch in the blue vertex. Indeed, that star must have  7 triangles, and only
$C_4$ has this property among legal stars. It is easy now to verify that $C_4$ does not fit.

We proceed now to hard cases $C_3,C_5,C_7$. The arguments 
are very similar to those used in the case of $C_5$ but the analysis 
is much more involved. Therefore we moved most of the proof to Appendix, as 
the proof of Claim~\ref{lmain} (page~\pageref{lmain}).

 \paragraph{The star of   $A$ in  $T$ is  $C_{3}$.}
The vertex  $A$ is colored green on  
Fig.~\ref{pi3011}(a) and its star is colored in
dark-grey. In Fig.~\ref{pi3011}(b) we show the 
neighborhood of that star (added tiles are colored in light-grey).
We can see that the neighborhood includes the large tile marked green in Fig.~\ref{star11}.
Now we have to show that the tiling $T$ does not 
the small triangle marked red in Fig.~\ref{pi3011}(c).
This is the statement of Claim~\ref{lmain}(a). 
\begin{figure}[ht]
   \begin{center}
\includegraphics[scale=1]{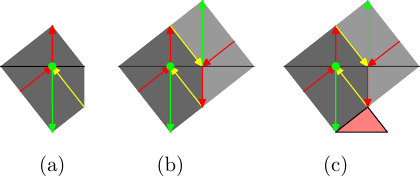}
\end{center}
 \caption{Composition of the star $C_3$.}\label{pi3011}
 \end{figure}

\paragraph{The star of the vertex $A$ is  $C_{4}$.}
The vertex $A$ is shown by the green point in Fig. \ref{pi4011}(a).
\begin{figure}[ht]
\begin{center}
\includegraphics[scale=1]{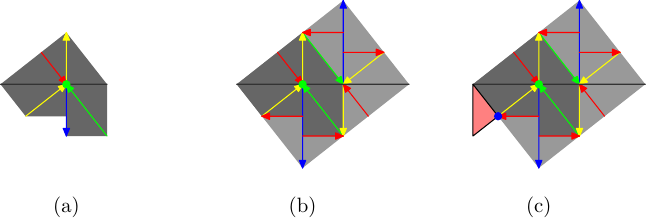}
\end{center}
\caption{Composition of the star $C_4$.}\label{pi4011}
\end{figure}
And 
Fig. \ref{pi411}(b) shows the neighborhood of that star.
We can see that the neighborhood includes both tiles marked green in Fig.~\ref{star11}.
It remains to show that the tiling $T$ does not 
the small triangle marked red (Claim~\ref{lmain}(b)).

\paragraph{The star of $A$ in $T$ is $C_{7}$.}
 Fig.~\ref{pi602}(a,b) show the star $C_{7}$   
and its neighborhood. We can see that the neighborhood includes all 6 tiles marked green in Fig.~\ref{star11},
except for the bottommost small tile.
\begin{figure}[ht]
\begin{center}
\includegraphics[scale=1]{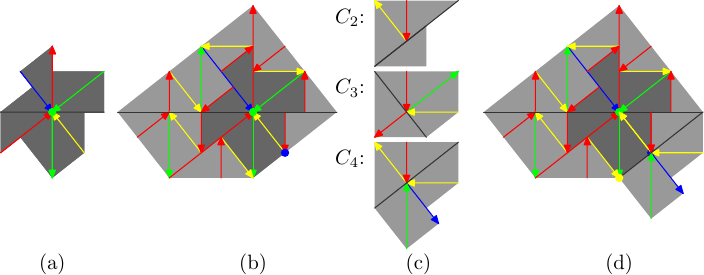}
\end{center}
\caption{Composition of the star $C_7$. 
}\label{pi602}
\end{figure}
To prove that it includes also that tile,
consider the blue vertex.
Only the stars  $C_2,C_3,C_4$
fit for the patch in that vertex, they are  shown on
Fig.~\ref{pi602}(c). If the star of the blue vertex is 
$C_3$, we are done, as that star includes the sought tile.
It remains to show that neither of the stars
$C_2,C_4$ can stand in the  blue vertex.

It is easy to show that 
$C_4$ cannot be there. Indeed, adding
that star to the patch, we obtain the patch shown
in Fig.~\ref{pi602}(d). No legal star fits for that patch
in the yellow vertex. And Claim~\ref{lmain}(c) claims that the star centered at the blue vertex cannot be
$C_2$ either. 
It remains to prove 
\begin{claim}\label{lmain}
The following patches (Fig.~\ref{pi1})
cannot occur in L-tilings of the plane. 
\end{claim}
\begin{figure}[ht]
   \begin{center}
\includegraphics[scale=1]{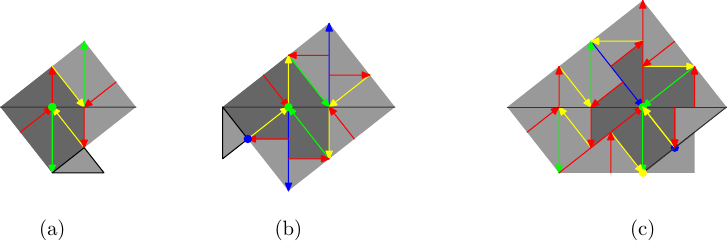}
\end{center}
\caption{The following patches 
do not occur in L-tilings of the plane (Claim~\ref{lmain})}\label{pi1}
\end{figure}
To prove the claim, we  explore a small neighborhood
of the patches in a similar way, as it was done in finding neighborhoods of the stars.
The  proof is deferred to Appendix.

\section{Acknowledgments} 
The author is grateful to Daria Pchelina  and Alexander Kozachinskii for verifying the proofs and reading the preliminary version of the paper,
and to anonymous referees for valuable comments.

\appendix
\section{Proof of Claim~\ref{lmain}}

(a) For the sake of the contradiction assume that the patch
in Fig.~\ref{pi301}(a) occurs in an L-tiling.
\begin{figure}[ht]
   \begin{center}
\includegraphics[scale=1]{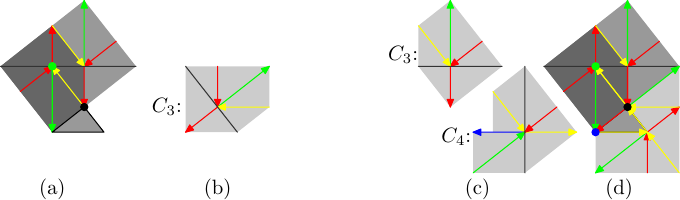}
\end{center}
 \caption{Composition of the star $C_3$ (the beginning).}\label{pi301}
 \end{figure}
Look at the vertex colored
black. 
A quick look at the list of legal stars reveals
that only the star $C_{3}$ (shown in Fig.~\ref{pi301}(b)) 
fits for the patch  in that vertex.  Adding that star and its neighborhood, we obtain the patch shown in Fig.~\ref{pi301}(d),
where the added tiles are colored light-gray.
 
Now look at the blue vertex on the bottom left. 
In that
vertex only the stars $C_3$ and $C_4$ fit.
Those stars are  shown 
in Fig.~\ref{pi301}(c).
The star $C_3$ has non-matching orientation of the green arrow, hence $C_3$ cannot
be the star within $T$ in the blue vertex. Therefore, it is the star
$C_4$. 
Adding the star
$C_4$   and its neighborhood, we obtain the tiling shown on 
 Fig.~\ref{pi302}(a).
\begin{figure}[ht]
 \begin{center}
\includegraphics[scale=1]{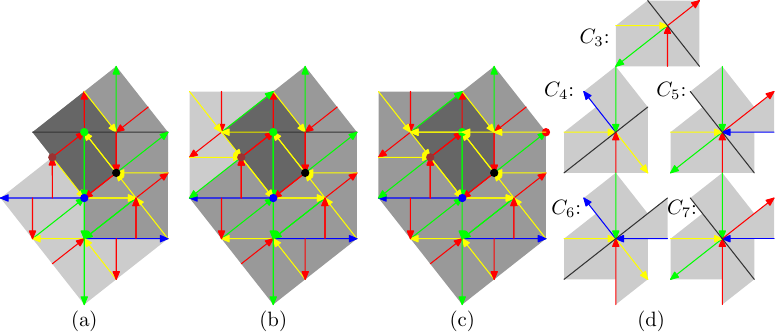}
\end{center}
\caption{Composition of the star $C_3$ (the end).}\label{pi302}
\end{figure}

A small search reveals that only the star
$C_3$ fits for that patch in the brown vertex on the top left.
Fig.~\ref{pi302}(b) shows the tiling that is obtained 
by adding the star $C_3$ and its neighborhood.
Thus we conclude that the axis of the initial star and its extension to the right must be
directed from the right to the left  (yellow arrows in Fig.~\ref{pi302}(c)).
Now look at the beginning of the leftmost
yellow arrow (the red point on the top right in Fig.~\ref{pi302}(c)).
Only the stars shown in Fig.~\ref{pi302}(d) fit for the resulting 
patch in the red vertex. However none of them can be there, since all they have non-matching orientation of 
the horizontal yellow arrow.

(b) 
For the sake of contradiction assume that an L-tiling $T$ includes the patch shown in 
Fig. \ref{pi412}(a).
\begin{figure}[ht]
\begin{center}
\includegraphics[scale=1]{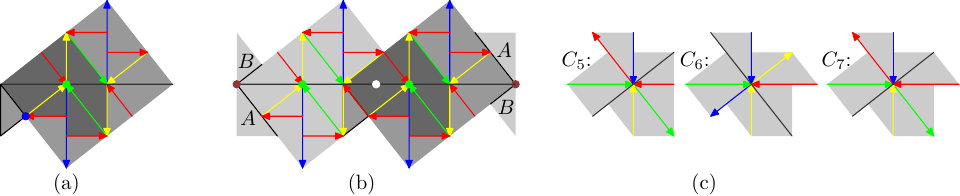}
\end{center}
\caption{Composition of the star $C_4$ (beginning)}\label{pi412}
\end{figure}
Our plan is the following.
We first show that, in addition to all tiles in the patch in
Fig.~\ref{pi412}(a), the tiling $T$ must contain 
all the tiles  shown in Fig.~\ref{pi412}(b).

Assume that this is done.
It is easy to verify that only the stars $C_5,C_6,C_7$ (shown in Fig.~\ref{pi412}(c))
fit for the patch in the rightmost brown
vertex. In all three cases  the axis of the initial star (the horizontal black line)
must be directed rightwards.
However, similar arguments
applied to the leftmost
brown vertex show that
 that axis must be directed leftwards, which is a contradiction.

So we have to show that the tiling $T$ must contain 
all the tiles  shown in Fig.~\ref{pi412}(b). To this end,
let us go back to  Fig. \ref{pi412}(a). We copied that patch in Fig. \ref{pi411}(a).
\begin{figure}[ht]
\begin{center}
\includegraphics[scale=1]{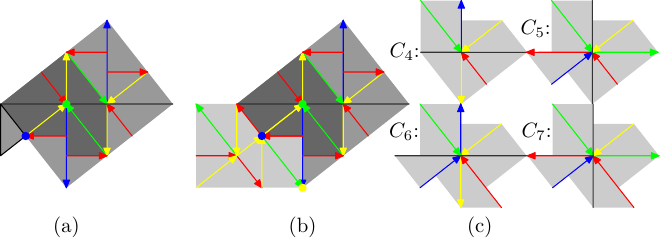}
\end{center}
\caption{Composition of the star $C_4$ (continued)}\label{pi411}
\end{figure}
Only the star  $C_3$ fits for this patch
in the blue vertex. Fig. \ref{pi411}(b)  shows the patch that is obtained by adding the star
$C_3$ and its neighborhood. 
Look now at the yellow vertex on the bottom.
Only the stars $C_4$,  $C_5$,  $C_6$, $C_7$ fit for the patch in that vertex, they
are shown in Fig. \ref{pi411}(c).
Note that the stars  $C_4,C_6$ (from the left column) have
non-matching orientation of the vertical blue arrow, which must direct downwards, hence cannot be there.
We will consider the remaining two cases separately.

 \emph{Case 1:  the star  $C_7$ is in the  yellow vertex in Fig. \ref{pi411}(b).}
 In this case we are able to derive a contradiction
quite easily. Adding the star 
$C_7$ and its  neighborhood we obtain the patch shown in Fig. \ref{pi434}(a).
\begin{figure}[ht]
\begin{center}
\includegraphics[scale=1]{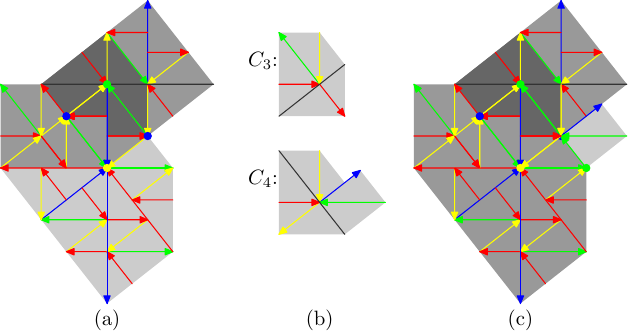}
\end{center}
\caption{Composition of the star $C_4$: case  1 (beginning)}\label{pi434}
\end{figure}
Only the stars $C_3,C_4$ (shown in Fig.~\ref{pi434}(b)) fit 
in the blue vertex.
Adding the star $C_4$, we obtain the patch shown in Fig. \ref{pi434}(c). 
We can see that no legal star fits in the green vertex on the right.

Hence only the star $C_3$ can be in the blue vertex.
Adding that star and its neighborhood,
we obtain the patch shown in Fig. \ref{pi402}(b). 
\begin{figure}[ht]
\begin{center}
\includegraphics[scale=1]{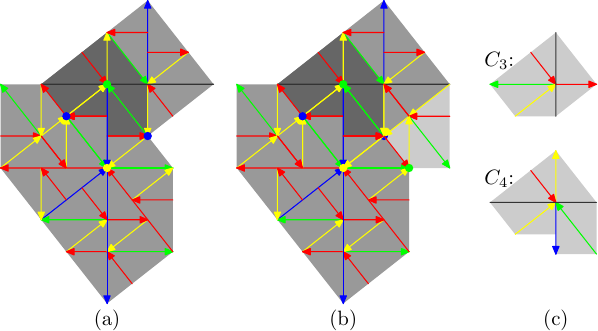}
\end{center}
\caption{Composition of the star $C_4$: case  1 (continued)}\label{pi402}
\end{figure}
Only the stars  $C_3,C_4$ (shown in Fig. \ref{pi402}(c))
fit for the resulting patch in the green vertex on the right. However the star
$C_3$ has non-matching orientation of the green arrow, hence
the star centered at the green vertex is 
$C_4$. Adding it and its neighborhood, we obtain the patch in Fig. \ref{pi403}(a). 
\begin{figure}[ht]
\begin{center}
\includegraphics[scale=1]{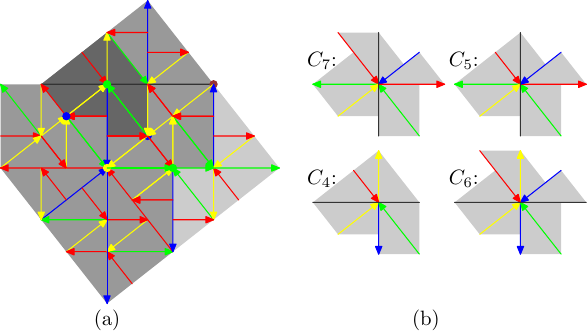}
\end{center}
\caption{Composition of the star $C_4$: case  1 (the end)}\label{pi403}
\end{figure}

Only the stars $C_4,C_5, C_6, C_7$ (shown in Fig. \ref{pi403}(b))
fit for the patch in the brown vertex on the top right.
However all those stars have non-matching orientation of the  yellow
arrow in the lower half of the star.
Thus we have derived a contradiction in the first case.

\emph{Case 2: the star centered at yellow vertex in Fig. \ref{pi411}(b)
  is $C_5$.}
In this case we need a more involved analysis. 
Let us go back to  Fig.~\ref{pi411}(b)
and add the star $C_5$ and its neighborhood
in the 
yellow vertex. We obtain the patch shown on 
Fig.~\ref{pi415}(b).
\begin{figure}[ht]
 \begin{center}
\includegraphics[scale=1]{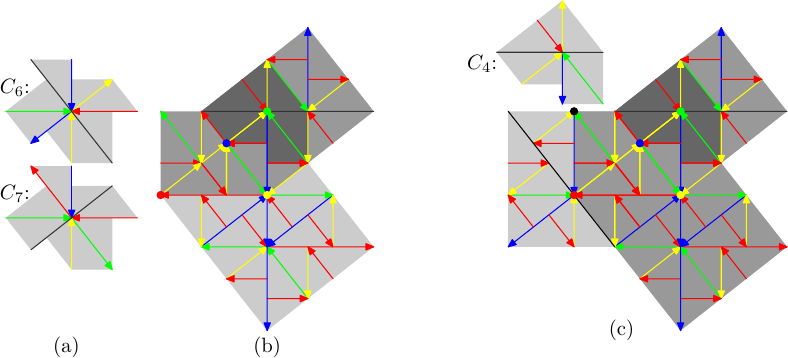}
\end{center}
 \caption{Composition of the star
   $C_4$: case 2 (beginning)}\label{pi415}
\end{figure}
Which stars fit for the patch in the leftmost vertex
(colored in red)? These are the stars 
$C_6,C_7$, which are  shown in Fig.~\ref{pi415}(a). Assume first that it is $C_6$.
Adding the neighborhood of the star $C_6$ centered at
the red point, we get the patch  shown in Fig.~\ref{pi415}(c).
If it is $C_7$, we get a patch that differs from
this one in orientation and labels of some sides.
This difference does not matter and therefore
we will consider only the case of  $C_6$.

Look at the vertex colored black (on the top left).
Only the star $C_4$ fits for
the patch in that vertex.  
Adding $C_4$ and its neighborhood in the black vertex we
obtain the patch shown in Fig.~\ref{pi408}(a).
\begin{figure}[ht]
\begin{center}
\includegraphics[scale=1]{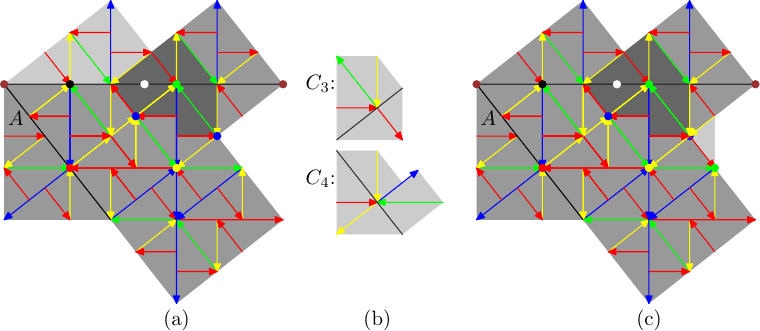}
\end{center}
\caption{Composition of the star $C_4$: case 2 (continued)}\label{pi408}
\end{figure}
We have shown that our tiling $T$
includes all the tiles from Fig.~\ref{pi412}(b) 
except for both triangles labeled by letter $B$ and one triangle labeled by letter $A$.
We have also shown that the tiling
$T$ includes the image of the initial star 
under the inversion through the white point.
Via central symmetrical arguments we can prove 
that $T$ includes also the other triangle labeled 
by letter $A$ in Fig.~\ref{pi412}(b).
It remains to show that $T$ includes both triangles labeled by 
$B$.

To this end look at the blue vertex on the right  in Fig.~\ref{pi408}(a).
Only the stars $C_3$ and $C_4$, shown in Fig.~\ref{pi408}(b),
fit for the patch in that vertex. If the star centered at the blue vertex is
$C_3$, we obtain the patch shown in Fig.~\ref{pi408}(c),
and no legal star fits for it
in the green vertex on the right.

In the remaining case the star centered at the blue vertex is  $C_4$,
and we get the patch shown in Fig. \ref{pi409}(a).
\begin{figure}[ht]
\begin{center}
\includegraphics[scale=1]{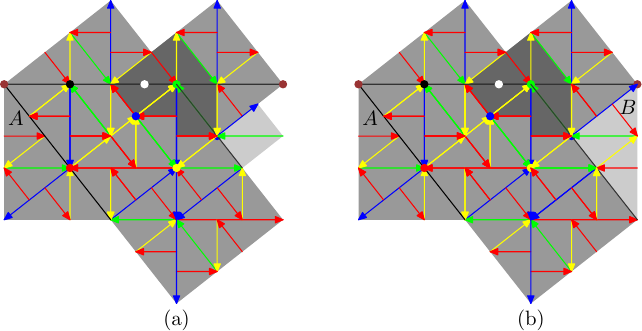}
 \end{center}
 \caption{Composition of the star $C_4$: case 2 (the end)}\label{pi409}
 \end{figure}
Adding its neighborhood, we get the patch that includes
the sought triangle $B$ (see Fig.~\ref{pi409}(b)). Via central symmetrical arguments we can prove 
that $T$ includes also the other triangle labeled 
by letter $B$.
We have reached our goal: we have proved
that the tiling $T$ includes the patch shown in Fig.~\ref{pi412}(b).

\newpage
(c) 
 For the sake of contradiction assume that an L-tiling includes the patch shown in Fig.~\ref{pi603}(a).
\begin{figure}[ht]
\begin{center}
\includegraphics[scale=1]{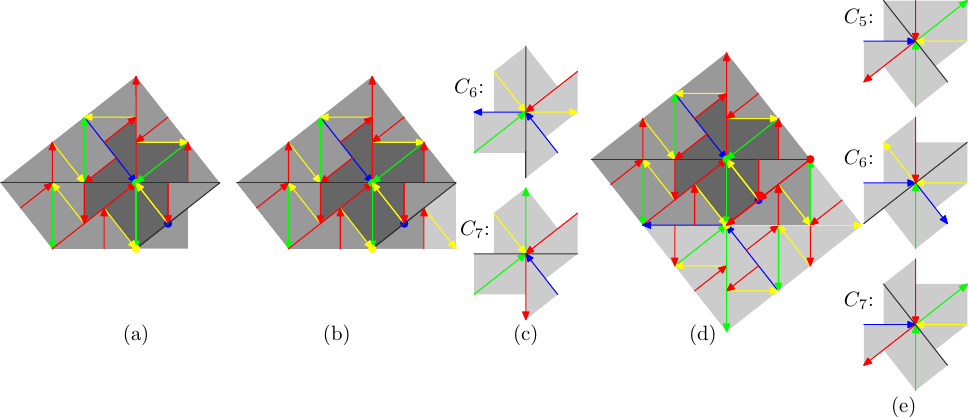}
\end{center}
\caption{Composition of the star $C_7$ (continued)} \label{pi603}
\end{figure}
First add the neighborhood of the star  $C_2$ centered at the blue vertex, we obtain Fig.~\ref{pi603}(b).
Only the stars
$C_6$ and 
$C_7$, shown on 
Fig.~\ref{pi603}(c), fit for the resulting patch
in the yellow vertex (on the bottom). However the star $C_7$
cannot be there, since its vertical green arrow has the non-matching
orientation.
Hence the star of the yellow vertex is $C_6$. Fig.~\ref{pi603}(d)
shows the patch which is obtained by adding that star
and its  neighborhood.
Now look at the red vertex on the right.
Only the stars  $C_5,C_6,C_7$ (Fig.~\ref{pi603}(e))
fit for the patch in that vertex. In all the three cases
there is a horizontal blue arrow that goes into
the red vertex. That arrow lies on the axis
of the initial star. 

Now we know the
color and orientation of that axis  (see Fig.~\ref{pi604}(a)).
\begin{figure}[ht]
\begin{center}
\includegraphics[scale=1]{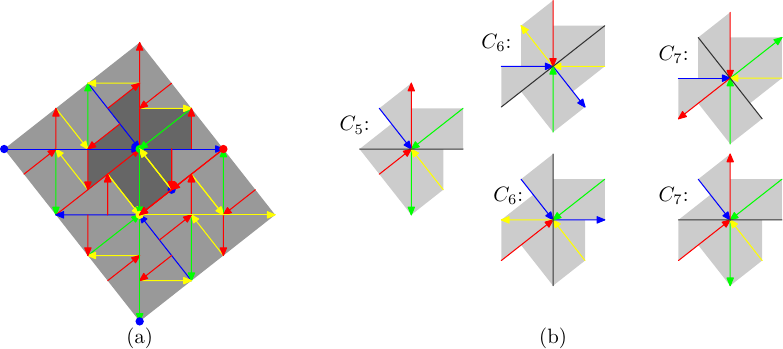}
\end{center}
\caption{Composition of the star $C_7$ (continued)}\label{pi604}
\end{figure}
We can find now the stars
in the leftmost and the bottommost
vertices (both are colored  blue).
Indeed, only the stars $C_5,C_6,C_7$
fit for the patch in the bottommost vertex,
and two latter stars
fit in two ways (see Fig.~\ref{pi604}(b)).
In four cases the vertical arrow has  
the red (and not green) color. Hence those cases are impossible
and only the lower star $C_6$ can stand in the
bottommost blue vertex.

A similar situation occurs in
the leftmost blue vertex: the  stars $C_3,C_4, C_5,C_6,C_7$ fit for the patch there
(two of them fit in two ways, see  Fig.~\ref{pi614}(a)).
\begin{figure}[ht]
\begin{center}
\includegraphics[scale=1]{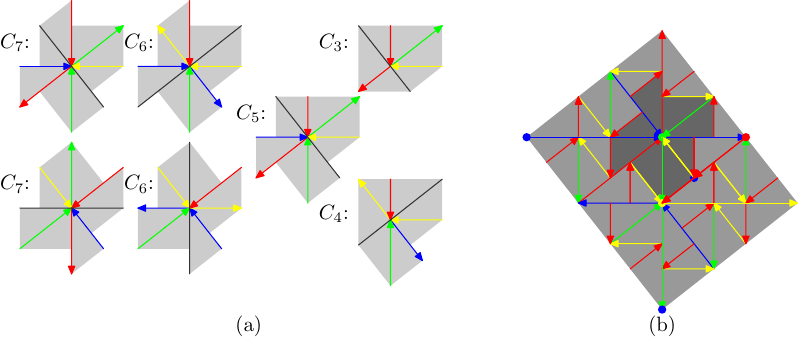}
\end{center}
\caption{Composition of the star $C_7$ (continued)}\label{pi614}
\end{figure}
However only one star, the lower $C_7$,
can have the matching color (blue) of the horizontal arrow.

Adding to the patch the stars in the blue
vertices and adding then their neighborhoods,
we obtain the patch shown in Fig.~\ref{pi605}(b).
\begin{figure}[ht]
\begin{center}
\includegraphics[scale=1]{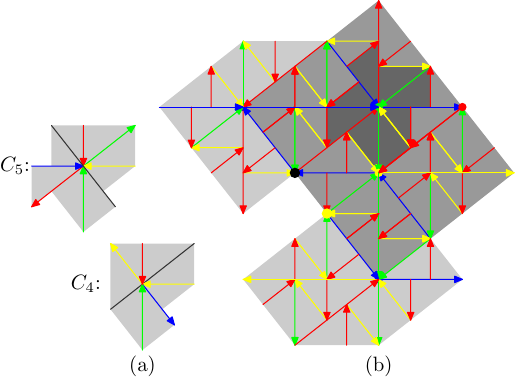}
\end{center}
\caption{Composition of the star $C_7$  (the end)}\label{pi605}
\end{figure}%
Two stars, $C_4$ and $C_5$, fit for the patch in the yellow vertex
(Fig.~\ref{pi605}(a)). In either case the star at the adjacent black vertex
would be illegal. If the star in the yellow vertex is $C_4$, then it has a
yellow arrow pointing to the north-west, into the black vertex; this arrow is continued
by a blue arrow, thus both arrows
lie on the axis of the star at the black vertex and make that axis two-colored. If the star in the yellow vertex is $C_5$, then,
after we add it, a blue arrow goes out of the black vertex to the south-east;
a quick look at the list of legal stars reveals that no legal star has two equally colored arrows
pointing in opposite directions.
In both  cases no legal
star can stand at the black vertex.
We have considered all the cases. The claim is proved.

\section{To cut out of paper}

\begin{figure}[ht]
       \begin{center}
\includegraphics[scale=1]{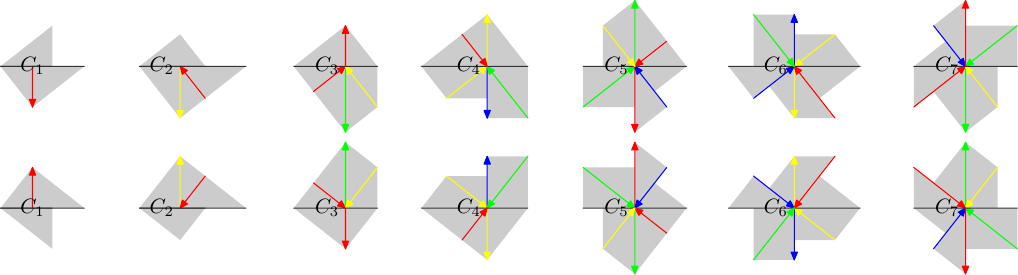}
        \end{center}
  \caption{Legal stars. 
The stars in the second row are reflections of the stars in the first row.
It is not necessary to use the color print, since labels on sides are used only
 in one place of the proof  (in the very end).}\label{pic32}
 \end{figure} 
 
  \begin{figure}[ht]
       \begin{center}
\includegraphics[scale=1]{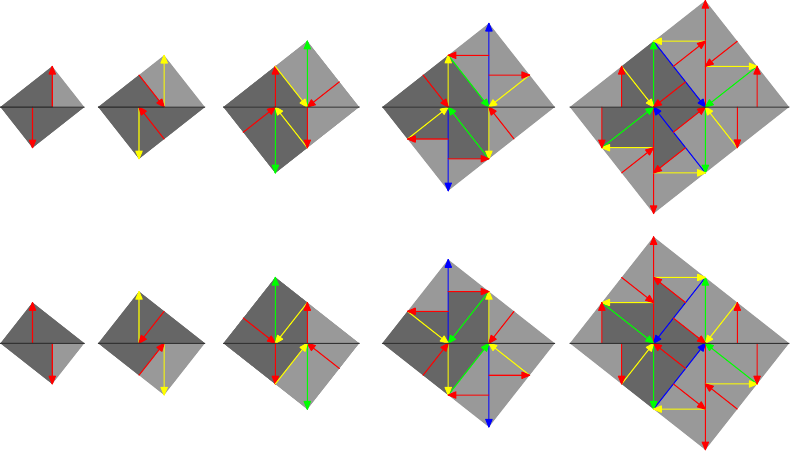}
        \end{center}
  \caption{The neighborhood of the first five stars.}\label{stars-neccessary10}
 \end{figure} 
 
 \newpage
  \begin{figure}[ht]
 \begin{center}
\includegraphics[scale=1]{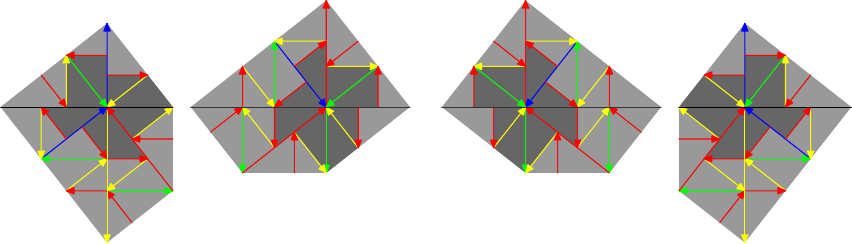}
\end{center}
    \caption{The neighborhoods of the stars
$C_6$, $C_7$.}\label{stars-neccessary1}
 \end{figure}
 
\end{document}